\RequirePackage[l2tabu, orthodox]{nag}
\documentclass[%a4paper,
11pt]{article}

\usepackage[inner=3.5cm,outer=3.5cm,top=2.5cm,bottom=3.2cm,headsep=0.5cm]{geometry}

\usepackage[final]{microtype}
\makeatletter
\def\MT@register@subst@font{\MT@exp@one@n\MT@in@clist\font@name\MT@font@list
   \ifMT@inlist@\else\xdef\MT@font@list{\MT@font@list\font@name,}\fi}
\makeatother 

\usepackage[small]{titlesec}

\usepackage{mparhack}

\usepackage{amsmath, amssymb, amsfonts, amsthm, mathtools}

\usepackage{fixmath}
\mathtoolsset{centercolon}

\usepackage[latin1]{inputenc}
\usepackage[T1]{fontenc}

\usepackage[draft=false,pdftex]{hyperref}
\hypersetup{
    unicode=false,          % non-Latin characters in Acrobat’s bookmarks
    pdftoolbar=true,        % show Acrobat’s toolbar?
    pdfmenubar=true,        % show Acrobat’s menu?
    pdffitwindow=false,     % window fit to page when opened
    pdfstartview={FitH},    % fits the width of the page to the window
    pdftitle={Bounding the size of an almost-equidistant set in Euclidean space},    % title
    pdfnewwindow=false,      % links in new window
    colorlinks=true,       % false: boxed links; true: colored links
    linkcolor=black,          % color of internal links
    citecolor=black,        % color of links to bibliography
    filecolor=black,      % color of file links
    urlcolor=black           % color of external links
}

\theoremstyle{plain}
\newtheorem{theorem}{Theorem}
\newtheorem{lemma}{Lemma}
\newtheorem{claim}{Claim}

\theoremstyle{definition}
\newcommand{\setbuilder}[2]{\left\{#1:#2\right\}}

\newcommand{\rank}[1]{\operatorname{rank}(#1)}

\newcommand{\norm}[1]{\left\lVert#1\right\rVert}

\newcommand{\Bignorm}[1]{\Bigl\lVert#1\Bigr\rVert}
\newcommand{\ipr}[2]{\left\langle #1, #2 \right\rangle}

\newcommand{\card}[1]{\left\lvert#1\right\rvert}

\newcommand{\numbersystem}[1]{\mathbb{#1}}
\newcommand{\bR}{\numbersystem{R}}

\newcommand{\define}[1]{%
    \emph{#1}%
}

\title{Bounding the size of an almost-equidistant set in Euclidean space}
\author{Andrey Kupavskii\thanks{Moscow Institute of Physics and Technology, Ecole Polytechnique F\'ed\'erale de Lausanne.\qquad\qquad Email: \texttt{kupavskii@yandex.ru}}
\and
Nabil H. Mustafa\thanks{Universit\'e Paris-Est, Laboratoire d'Informatique Gaspard-Monge, ESIEE Paris, France. \qquad\quad Email: \texttt{mustafan@esiee.fr}}\;\,\thanks{The work of Nabil H. Mustafa
in this paper has been supported by the grant ANR SAGA (JCJC-14-CE25-0016-01).}
\and
Konrad J. Swanepoel\thanks{Department of Mathematics, London School of Economics and Political Science, London. \qquad\quad Email: \texttt{k.swanepoel@lse.ac.uk}}}
\date{}

\begin{document}
\maketitle

\begin{abstract}
A set of points in $d$-dimensional Euclidean space is \define{almost equidistant} if among any three points of the set, some two are at distance $1$.
We show that an almost-equidistant set in $\bR^d$ has cardinality $O(d^{4/3})$.
\end{abstract}

\section{Introduction}
A set of lines through the origin of Euclidean $d$-space $\bR^d$ is \define{almost orthogonal} if among any three of the lines, some two are orthogonal.
Erd\H{o}s asked (see \cite{Rosenfeld}) what is the largest cardinality of an almost-orthogonal set of lines in $\bR^d$?
By taking two sets of $d$ pairwise orthogonal lines, we see that $2d$ is possible.
Rosenfeld \cite{Rosenfeld} showed that $2d$ is the maximum using eigenvalues.
This result was subsequently given simpler proofs by Pudl\'ak \cite{Pudlak} and Deaett \cite{Deaett}.

In this note we consider the following analogous notion, replacing orthogonal pairs of lines by pairs of point at unit distance.
A subset $V$ of Euclidean $d$-space $\bR^d$ is \define{almost equidistant} if among any three points in $V$, some two are at Euclidean distance $1$.
We investigate the largest size, which we denote by $f(d)$, of an almost-equidistant set in $\bR^d$.
Although this is a very natural question, it seems to be much harder than the question of Erd\H{o}s, which can be refomulated as asking for the largest size of an almost-equidistant set on a sphere of radius $1/\sqrt{2}$ in $\bR^d$.
Before stating our main result, we give an overview of what is known about the size of $f(d)$.

Bezdek, Nasz\'odi and Visy \cite{Bezdek-Naszodi-Visy} showed that $f(2)\leq 7$, and
Istv\'an Talata (personal communication, 2007) showed that the only almost-equidistant set in $\bR^2$ with $7$ points is the Moser spindle.
Gy\"orey \cite{Gyorey2004} showed that $f(3)\leq 10$ and that there is a unique almost-equidistant set of $10$ points in $\bR^3$, a configuration originally considered by Nechushtan \cite{Nechushtan}.
The Moser spindle can be generalized to higher dimensions, giving an almost-equidistant set of $2d+3$ points in $\bR^d$ \cite{Bezdek-Langi}.
(We mention that Bezdek and Langi \cite{Bezdek-Langi} considered the variant of Erd\H{o}s's problem where the radius of the sphere is arbitrary instead of $1/\sqrt{2}$.)
A construction of Larman and Rogers \cite{Larman-Rogers} shows that $f(5)\geq 16$.
Since there does not exist a set of $d+2$ points in $\bR^d$ that are pairwise at distance $1$, it follows that $f(d)\leq R(d+2,3)-1$, where the Ramsey number $R(a,b)$ is the smallest $n$ such that whenever each edge of the complete graph on $n$ vertices is coloured blue or red, there is either a blue clique of size $a$ or a red clique of size $b$.
Ajtai, Koml\'{o}s, and Szemer\'{e}di~\cite{ajKomSze80} proved $R(k,3) = O(k^2/\log{k})$, which implies the asymptotic upper bound $f(d) \leq O(d^2/\log{d})$.
Balko et al.\ \cite{BPSSV} generalized the Nechushtan configuration to higher dimensions, giving $f(d)\geq 2d+4$ for all $d\geq 3$.
They also obtained the asymptotic upper bound $f(d)=O(d^{3/2})$ by an argument based on Deaett's paper \cite{Deaett}.
Using computer search and ad hoc geometric arguments, they obtained the following bounds for small $d$: $f(4)\leq 13$, $f(5)\leq 20$, $18\leq f(6)\leq 26$, $20\leq f(7)\leq 34$, and $f(9)\geq f(8)\geq 24$.
Polyanskii \cite{Polyanskii} subsequently improved the asymptotic upper bound to $f(d)=O(d^{13/9})$.

In this note we obtain a further improvement to the upper bound.
\begin{theorem}\label{thm1}
An almost-equidistant set in $\bR^d$ has cardinality $O(d^{4/3})$.
\end{theorem}
Its proof is based on the approach of Balko et al.\ \cite{BPSSV}.
Before giving the proof of Theorem~\ref{thm1} in Section~\ref{section:proof},
we collect some lemmas in the next section.

\section{Preliminaries}
We call a finite subset $C$ of $\bR^d$ (the vertex set of) a \define{unit simplex} if the distance between any two points in $C$ equals $1$.
It is well known that if $C$ is a unit simplex then $\card{C}\leq d+1$.
Given any finite $V\subset\bR^d$, we define the \define{unit-distance graph $G=(V,E)$ on $V$} to be the graph with $vw\in E$ iff $\norm{v-w}=1$.
Thus, $C\subset V$ is a unit simplex iff it is a clique in $G$.
We denote the set of neighbours of $v\in V$ in $G$ by $N(v)$.

The following well-known lemma gives a lower bound for the rank of a square matrix in terms of its entries~\cite{Alon, Deaett, Pudlak}.
\begin{lemma}\label{lemma0}
For any non-zero $n\times n$ symmetric matrix $A=[a_{i,j}]$,
\[\rank A\geq \frac{(\sum_i a_{i,i})^2}{\sum_{i,j} a_{i,j}^2}.\]
\end{lemma}
\begin{lemma}\label{lemma1}
Let $C$ be a unit simplex with centroid $c=\frac{1}{\card{C}}\sum_{v\in C}v$.
Then \[\norm{v-c}^2=\frac12\left(1-\frac1{\card{C}}\right) \quad\text{for all $v\in C$,}\] and \[\ipr{v-c}{v'-c}=-\frac{1}{2\card{C}}\quad\text{for all distinct $v,v'\in C$.}\]
\end{lemma}
\begin{proof}
We may translate $C$ so that $c$ is the origin $o$.
Write $C=\{p_1,\dots,p_k\}$.
By symmetry, $\alpha:=\norm{p_i}^2$ is independent of $i$, and $\beta:=\ipr{v_i}{v_j}$ ($i\neq j$) is independent of $i$ and $j$.
Then
\[ 0 = \norm{\sum_{i=1}^k p_i}^2 = k\alpha +k(k-1)\beta\]
and \[ 1 = \norm{p_i-p_j}^2=2\alpha-2\beta.\]
Solving these two linear equations in $\alpha$ and $\beta$, we obtain $\alpha=\frac12-\frac{1}{2k}$ and $\beta=-\frac{1}{2k}$.
\end{proof}
\begin{lemma}\label{lemma2}
Let $C$ be a unit simplex with centroid $c=\frac{1}{\card{C}}\sum_{v\in C}v$, and let $F\subset C$ with centroid $f=\frac{1}{\card{F}}\sum_{v\in F}v$.
Then \[\norm{c-f}^2 = \frac12\left(\frac{1}{\card{F}}-\frac{1}{\card{C}}\right).\]
\end{lemma}
\begin{proof}
Let $k=\card{C}$, $\ell=\card{F}$.
Then
\begin{align*}
\norm{f-c}^2 &= \norm{\frac{1}{\ell}\sum_{v\in F}(v-c)}^2\\
&= \frac{1}{\ell^2}\left(\ell\cdot\frac12(1-\frac{1}{k})-\frac{\ell(\ell-1)}{2k}\right)\quad\text{by Lemma~\ref{lemma1}}\\
&= \frac12\left(\frac{1}{\ell}-\frac{1}{k}\right). \qedhere
\end{align*}
\end{proof}

\section{Proof of Theorem~\ref{thm1}}\label{section:proof}
Let $G$ be the unit-distance graph of the almost-equidistant set $V$.
Thus, the complement of $G$ is $K_3$-free, and the non-neighbours of any vertex form a unit simplex.
Let $C$ be a clique of maximum cardinality in $G$.
Write $k=\card{C}$.
Each $v\in V\setminus C$ is a non-neighbour of some point in $C$, and it follows that $\card{V}\leq\card{C}+\card{C}k=k^2+k$.
Thus, without loss of generality, $k> d^{2/3}$.

We split $V$ up into two parts, each to be bounded separately.
Let \[N=\setbuilder{v\in V}{\card{N(v)\cap C}\geq k-k^{4/3}d^{-2/3}}.\]
Note that $k^{4/3}d^{-2/3}=O(d^{-1/3})k$.
We first bound the complement of $N$.
Consider the set \[X=\setbuilder{(u,v)\in C\times V\setminus N}{uv\notin E(G)}.\]
For each $v\in V\setminus N$, there are more than $k^{4/3}d^{-2/3}$ points $u\in C$ such that $u\notin N(v)$, hence $\card{X}>k^{4/3}d^{-2/3}\card{V\setminus N}$.
On the other hand, for each $u\in C$, the set of non-neighbours of $u$ forms a clique, so has cardinality at most $k$, and $\card{X}\leq\card{C}k=k^2$.
It follows that
\begin{equation}\label{eq1}
\card{V\setminus N} < k^{2/3}d^{2/3}.
\end{equation}
Next we estimate $\card{N}$.
Without loss of generality, $\frac{1}{k}\sum_{v\in C}v=o$ and $N=\{v_1,\dots,v_n\}$.
We want to apply Lemma~\ref{lemma0} to the $n\times n$ matrix $A=[\ipr{v_i}{v_j}]$, which has rank at most $d$.
\begin{claim}\label{claim1}
For each $i=1\dots,n$, $\norm{v_i}^2=\frac12+O(k^{-1/3}d^{-1/3})$, and for each $v_iv_j\in E(G)$, $\ipr{v_i}{v_j}=O(k^{-1/3}d^{-1/3})$.
\end{claim}
\begin{proof}[Proof of Claim~\ref{claim1}]
Let $C_i:=N(v_i)\cap C$ and $k_i:=\card{C_i}\geq k-k^{4/3}d^{-2/3}$,
and $c_i:=\frac{1}{k_i}\sum_{v\in C_i}v$.
For any $v\in C_i$, $\norm{v_i-c_i}^2 = \norm{v_i-v}^2 - \norm{v-c_i}^2$.
By Lemma~\ref{lemma1}, 
$\norm{v-c_i}^2=\frac12\left(1-\frac{1}{k_i}\right)$, hence
\[\norm{v_i-c_i}=\sqrt{\frac12\left(1+\frac{1}{k_i}\right)}=\frac{1}{\sqrt{2}}+O(k^{-1}).\]
By Lemma~\ref{lemma2}, \[ \norm{c_i}=\sqrt{\frac12\left(\frac{1}{k_i}-\frac{1}{k}\right)}\leq\sqrt{\frac12\left(\frac{1}{k-k^{4/3}d^{-2/3}}-\frac{1}{k}\right)} = O(k^{-1/3}d^{-1/3}).\]
By the triangle inequality,
\[ \norm{v_i} = \norm{v_i-c_i}+O(\norm{c_i}) = \frac{1}{\sqrt{2}} + O(k^{-1/3}d^{-1/3}),\]
and $\norm{v_i}^2=\frac12+O(k^{-1/3}d^{-1/3})$.
Also, $2\ipr{v_i}{v_j}=\norm{v_i}^2+\norm{v_j}^2-1=O(k^{-1/3}d^{-1/3})$.
\end{proof}
\begin{claim}\label{claim2}
For each $i=1,\dots,n$, \[\sum_{\substack{j=1\\ v_iv_j\notin E(G)}}^n\ipr{v_i}{v_j}^2=O(k^{2/3}d^{-1/3}).\]
\end{claim}
\begin{proof}[Proof of Claim~\ref{claim2}]
The non-neighbours $N\setminus N(v_i)$ of $v_i$ form a unit simplex with cardinality $t:=\card{N\setminus N(v_i)}\leq k$ and with centroid $c$, say.
If $t=d+1$, remove one point $v_j$ from the unit simplex, which decreases the sum by $\ipr{v_i}{v_j}^2=O(1)$.
Thus, without loss of generality, $t\leq d$, and there exists a point $p\in\bR^d$ such that $p-c$ is orthogonal to the affine hull of $N(v_i)$, $\norm{p-c}=1/\sqrt{2t}$, the set $\setbuilder{v_j-p}{v_j\in N\setminus N(v_i)}$ is orthogonal, and $\norm{v_j-p}=1/\sqrt{2}$ for each non-neighbour $v_j$ of $v_i$.
Then, by the finite Bessel inequality, \[\sum_{v_j\in N\setminus N(v_i)}\ipr{v_i}{v_j-p}^2\leq\frac12\norm{v_i}^2,\] hence by applying Cauchy-Schwarz a few times,
\begin{align}
\sum_{v_j\in N\setminus N(v_i)}\ipr{v_i}{v_j}^2 &= \sum_{v_j\in N\setminus N(v_i)}\bigl(\ipr{v_i}{v_j-p} + \ipr{v_i}{p-c} + \ipr{v_i}{c}\bigr)^2\notag\\
&\leq 3\sum_{v_j\in N\setminus N(v_i)}\bigl(\ipr{v_i}{v_j-p}^2 + \ipr{v_i}{p-c}^2 + \ipr{v_i}{c}^2\bigr)\notag\\
&\leq 3\left(\frac12\norm{v_i}^2 + t\norm{v_i}^2\norm{p-c}^2+t\norm{v_i}^2\norm{c}^2\right)\notag\\
&\leq 3(\norm{v_i}^2+t\norm{v_i}^2\norm{c}^2). \label{eq2}
\end{align}
By Claim~\ref{claim1}, $\norm{v_i}^2=\frac12+O(k^{-1/3}d^{-1/3})$ and
\begin{align*}
\norm{c}^2 &= \Bignorm{\frac1t\sum_{v_j\in N\setminus N(v_i)} v_j}^2 = \frac1{t^2}\biggl(\,\sum_{v_j\in N\setminus N(v_i)} \norm{v_j}^2 + \sum_{\substack{v_j,v_{j'}\in N\setminus N(v_i)\\ v_j\neq v_{j'}}}\ipr{v_j}{v_{j'}}\biggr)\\
&\leq \frac{1}{t^2}\left(t\left(\frac12+O(k^{-1/3}d^{-1/3})\right) + t(t-1)O(k^{-1/3}d^{-1/3})\right)\\
&= \frac{1}{2t} + O(k^{-1/3}d^{-1/3}).
\end{align*}
Therefore, \[t\norm{c}^2=\frac12+O(tk^{-1/3}d^{-1/3}) = O(k^{2/3}d^{-1/3}).\]
Substitute this back into \eqref{eq2} to finish the proof of Claim~\ref{claim2}.
\end{proof}
We now finish the proof of the theorem.
By Claim~\ref{claim2},
\begin{align*}
\sum_{j=1}^n\ipr{v_i}{v_j}^2 &= \norm{v_i}^4 + \sum_{v_j\in N(v_i)} \ipr{v_i}{v_j}^2 + O(k^{2/3}d^{-1/3}) \\
&= nO(k^{-2/3}d^{-2/3}) + O(k^{2/3}d^{-1/3}) \quad\text{by Claim~\ref{claim1}.}
\end{align*}
Also by Claim~\ref{claim1}, $\sum_{i=1}^n\norm{v_i}^2 = \Omega(n)$.
Therefore, by Lemma~\ref{lemma0},
\[ d\geq \rank A\geq \frac{\left(\sum_{i=1}^n\norm{v_i}^2\right)^2}{\sum_{i,j=1}^n\ipr{v_i}{v_j}^2} = \frac{\Omega(n^2)}{n\bigl(nO(k^{-2/3}d^{-2/3}) + O(k^{2/3}d^{-1/3})\bigr)},\]
hence $n=O(nk^{-2/3}d^{1/3}) + O(k^{2/3}d^{2/3})$.
Since $O(k^{-2/3}d^{1/3})=o(1)$, it follows that $\card{N}=n=O(k^{2/3}d^{2/3})$.
Recalling \eqref{eq1}, we obtain that \[\card{V}=\card{N}+\card{V\setminus N} = O(k^{2/3}d^{2/3}) = O(d^{4/3}).\qedhere\]

\providecommand{\bysame}{\leavevmode\hbox to3em{\hrulefill}\thinspace}
\providecommand{\MR}{\relax\ifhmode\unskip\space\fi MR }
\providecommand{\MRhref}[2]{%
  \href{http://www.ams.org/mathscinet-getitem?mr=#1}{#2}
}
\providecommand{\href}[2]{#2}


\begin{thebibliography}{10}

\bibitem{ajKomSze80}
Mikl\'{o}s Ajtai, J\'{a}nos Koml\'{o}s, and Endre Szemer\'{e}di, \emph{A note
  on {R}amsey numbers}, J. Combin.\ Theory Ser.\ A \textbf{29} (1980), no. 3, 354--360.

\bibitem{Alon}
Noga Alon, \emph{Problems and results in extremal combinatorics. {I}}, Discrete
  Math.\ \textbf{273} (2003), no.~1-3, 31--53. %EuroComb'01 (Barcelona).

\bibitem{BPSSV}
Martin Balko, Attila P\'or, Manfred Scheucher, Konrad Swanepoel, and Pavel Valtr,
\emph{Almost-equidistant sets}, 2017,
\url{http://arxiv.org/abs/1706.06375}

\bibitem{Bezdek-Langi}
K{\'a}roly Bezdek and Zsolt L{\'a}ngi, \emph{Almost equidistant points on
  {$S^{d-1}$}}, Period.\ Math.\ Hungar.\ \textbf{39} (1999), no.~1-3, 139--144.
  %Discrete geometry and rigidity (Budapest, 1999).

\bibitem{Bezdek-Naszodi-Visy}
K\'aroly Bezdek, M\'arton Nasz\'odi, and Bal\'azs Visy, \emph{On the {$m$}th
  {P}etty numbers of normed spaces}, Discrete geometry, Monogr.\ Textbooks Pure
  Appl.\ Math., vol.~253, Dekker, New York, 2003, pp.~291--304.
  
\bibitem{Deaett}
Louis Deaett, \emph{The minimum semidefinite rank of a triangle-free graph},
  Linear Algebra Appl.\ \textbf{434} (2011), no.~8, 1945--1955.

\bibitem{Gyorey2004}
Bernadett Gy\"orey, \emph{Diszkr\'et metrikus terek be\'agyaz\'asai}, Master's
  thesis, E\"otv\"os Lor\'and University, Budapest, 2004, in Hungarian.

\bibitem{Larman-Rogers}
David~G. Larman and Claude~A. Rogers, \emph{The realization of distances within
  sets in {Euclidean} space}, Mathematika \textbf{19} (1972), 1--24.
  
\bibitem{Nechushtan}
Oren Nechushtan, \emph{On the space chromatic number}, Discrete Math.\
  \textbf{256} (2002), no.~1--2, 499--507.

\bibitem{Polyanskii}
Alexandr~Polyanskii, \emph{On almost-equidistant sets}, 2017, 
  \url{http://arxiv.org/abs/1707.00295}

\bibitem{Pudlak}
Pavel Pudl{\'a}k, \emph{Cycles of nonzero elements in low rank matrices},
  Combinatorica \textbf{22} (2002), no.~2, 321--334.
  %Special issue: Paul Erd{\H{o}}s and his mathematics.

\bibitem{Rosenfeld}
Moshe Rosenfeld, \emph{Almost orthogonal lines in {$E^d$}}, Applied geometry
  and discrete mathematics. {The} {Victor} {Klee} {Festschrift} (Peter
  Gritzmann and Bernd Sturmfels, eds.), DIMACS Ser. Discrete Math.\ Theoret.\
  Comput.\ Sci., vol.~4, Amer.\ Math.\ Soc., Providence, RI, 1991, pp.~489--492.

\end{thebibliography}
\end{document}